\documentclass[11pt]{amsart}

\usepackage[T2A]{fontenc}
\usepackage[utf8]{inputenc}
\usepackage[english]{babel}
\usepackage{amsfonts, amssymb}
\usepackage{ccfonts,euler}
\usepackage{stmaryrd}
\usepackage[left=2.0cm,right=2.0cm,top=2.0cm,bottom=2.0cm]{geometry}
\usepackage{cite}%
\usepackage[colorlinks=true]{hyperref}

\usepackage[all,cmtip]{xy}
\usepackage{capt-of} 
\usepackage{thm-restate}

\title{Overgroups of exterior powers of an elementary group.\\Levels}

\author{Roman~Lubkov}
\thanks{The main results of the present paper were proven in the framework of the RSF project 17-11-01261}
\address[Roman~Lubkov]{St.~Petersburg Department of V.\,A.\,Steklov Institute of Mathematics of the Russian Academy of Sciences}
\email{RomanLubkov@yandex.ru}

\author{Ilia~Nekrasov}
\address[Ilia~Nekrasov]{Department of Mathematics, University of Michigan, Ann Arbor, MI}
\email{inekras@umich.edu, geometr.nekrasov@gmail.com}


\date{}
\let\opn\operatorname
\DeclareMathOperator{\E}{E}
\DeclareMathOperator{\M}{M}
\DeclareMathOperator{\CC}{C}
\DeclareMathOperator{\EO}{EO}
\DeclareMathOperator{\SL}{SL}
\DeclareMathOperator{\GL}{GL}
\DeclareMathOperator{\sign}{sgn}
\DeclareMathOperator{\height}{ht}
\DeclareMathOperator{\Ar}{Ar}

\renewcommand{\trianglelefteq}{\trianglelefteqslant}
\renewcommand{\leq}{\leqslant}
\renewcommand{\geq}{\geqslant}

\newcommand{\bw}[1]{\mathord{\raisebox{2pt}
{\hbox{$\scriptstyle{\bigwedge^{\!#1}}$}}}}
\newcommand\blank{\mathord{\hbox to 1.5ex{\hrulefill}}\,}

\theoremstyle{plain}
\newtheorem{theorem}{Theorem}
\newtheorem*{generalproblem}{Problem}
\newtheorem{proposition}[theorem]{Proposition}
\newtheorem{lemma}[theorem]{Lemma}
\newtheorem{corollary}[theorem]{Corollary}
\theoremstyle{remark}
\newtheorem*{remark}{Remark}

\begin{document}

\begin{abstract}
We prove a first part of the \textit{standard description} of groups $H$ lying between an exterior power of an elementary group $\bw{m}\E_n(R)$ and a general linear group $\GL_{\binom{n}{m}}(R)$ for a commutative ring $R, 2 \in R^{*}$ and $n \geq 3m$. The description uses the classical notion of a \textit{level}: for every group $H$ we find a unique ideal $A$ of the ground ring $R$ which describes $H$.
\end{abstract}
\maketitle

\section*{Introduction}\label{sec:fundrepres}

The present paper is devoted to the solution of the following general problem.
\begin{generalproblem}
Let $R$ be an arbitrary commutative associative ring with $1$ and let $\Phi$ be a reduced irreducible root system. $G(\Phi, \blank)$ is a Chevalley--Demazure group scheme and $\rho\colon G(\Phi, \blank) \longrightarrow \GL_{N}(\blank)$ is its arbitrary representation. Describe all overgroups $H$ of the elementary subgroup $\E_{G}(\Phi, R)$ in the representation $\rho$:
$$\E_{G, \rho}(\Phi, R ) \leq H \leq \GL_{N}(R).$$
\end{generalproblem}

The conjectural answer, the \textit{standard overgroup description}, in a general case can be formulated as follows.
For any overgroup $H$ of the elementary group there exists a net of ideals $\mathbb{A}$ of the ring $R$ such that
$$\E_{G,\rho}(\Phi,R)\cdot\E_N(R,\mathbb{A})\leq H \leq N_{\GL_{N}(R)}\big(\E_{G,\rho}(\Phi, R)\cdot\E_N(R,\mathbb{A})\big),$$
where $\E_N(R,\mathbb{A})$ is a \textit{relative elementary subgroup} for the net $\mathbb{A}$.

In the special case of a trivial net, i.\,e., $\mathbb{A} = \{A\}$ consists of one ideal $A$ of the initial ring $R$, overgroups of the group $\E_{G}(\Phi,R)$ can be parametrized by the ideal $A$ of the ring $R$:
\begin{equation}
\E_{G,\rho}(\Phi,R)\cdot\E_N(R,A)\leq H \leq N_{\GL_{N}(R)}\big(\E_{G,\rho}(\Phi, R)\cdot\E_N(R,A)\big),
    \label{eq:G_A}
\end{equation}
where $\E_N(R,A)$ equals $\E_N( A)^{\E_N(R)}$ by definition. 
\medskip

Based on the classification of finite simple groups in 1984, Michael Aschbacher proved the Subgroup structure theorem~\cite{AschbacherClasses}. It states that every maximal subgroup of a finite classical group either falls into one of the eight explicitly described classes $\mathcal{C}_1$--$\mathcal{C}_8$, or is an ``almost'' simple group in an irreducible representation (class $\mathcal{S}$). In the recent past, many experts studied overgroups of groups from the Aschbacher classes for some special cases of fields. For finite fields and algebraically closed fields maximality of subgroups was obtained by Peter Kleidman and Martin Liebeck, see~\cite{KleidLiebSubStruct,LiebSeitStrucClass}. Oliver King, Roger Dye, and Shang Zhi Li proved maximality of groups from Aschbacher classes for arbitrary fields or described its overgroups in cases where they are not maximal, see~\cite{KingDiag,KingUnit,KingOrthog,DyeSpO,DyeOinSp,DyeSubPolarities,ShangZhiGL,ShangZhiUnit,ShangZhiSUOmega}. We recommend the surveys~\cite{VavSubgroup90,VavSbgs,VavStepSurvey}, which contain necessary preliminaries, complete history, and known results about the initial problem.

In the present paper, we consider the case of the $m$-th fundamental representation of a simply connected group of type $A_{n-1}$, i.\,e., the scheme $G_{\rho}(\Phi,\blank)$ equals a [Zariski] closure of the affine group scheme $\SL_{n}(\blank)$ in the representation with the highest weight $\varpi_{m}$. In our case the extended Chevalley group scheme coincides with the $m$-th fundamental representation of the general linear group scheme $\GL_{n}(\blank)$. This case corresponds to the Aschbacher class ${\mathcal{S}}$ consisting of almost simple groups in certain absolutely irreducible representations. Morally, the paper is a continuation of a series of papers by the St.~Petersburg school on subgroups in classical groups over a commutative ring, see~\cite{BV82,BV84,VavSubgroup90,VP-EOeven,VP-EOodd,VP-Ep,PetOvergr,StepGL,StepNonstandard,LuzE6E7-GL,LuzF4-E6,VavLuzgE6,VavLuzgE7,AnaVavSinI}.

We deal only with the case of a trivial net, i.\,e., a net consists of only one ideal $A$. As shown below (Propositions~\ref{prop:InversInclus} and~\ref{prop:LevelsEqualGeneral}), it imposes a constraint $n\geq 3m$, we proceed with this restriction. In this case the general answer has the following form. Let $N = \binom{n}{m}$ and $H$ be a subgroup in $\GL_{N}(R)$ containing $\bw{m}\E_n(R)$. Then there exists a unique maximal ideal $A\trianglelefteq R$ such that
$$\bw{m}\E_n(R)\cdot\E_N(R,A)\leq H \leq N_{\GL_{N}(R)}\big(\bw{m}\E_n(R)\cdot\E_N(R,A)\big).\eqno{(*)}$$
The present paper is the first part in the serial study of the problem. We construct a level and calculate a normalizer of connected (i.\,e., perfect) intermediate subgroups. Further, it is necessary to construct invariant forms for $\bw{m}\SL_n(R)$ and calculate a normalizer of $\bw{m}\E_n(R)$. Finally, we will extract an elementary transvection from an intermediate subgroup $H$. These steps are enough to solve the problem completely, see~\cite{LubStepSub}.

There are separate results for special cases of the ring $R$ such as a finite field $K$ or an algebraically closed field. For finite fields Bruce Cooperstein proved maximality of the normalizer $N_G\bigl(\bw{2}\E_n(K)\bigr)$ in $\GL_N(K)$~\cite{CoopersteinStructure}. For algebraically closed fields a description of overgroups of $\bw{m}\E_n(K)$ follows from the classical results about maximal subgroups of classical algebraic groups, for instance, see~\cite{SeitzMaxSub}.
\medskip

The present paper is organized as follows. In the next Section we formulate main results of the paper. In Section~\ref{sec:princnot} we set up the notation. Section~\ref{sec:proofs} contains all complete proofs, for instance, in~\S\ref{subsec:square} we have considered an important special case --- a level computation for exterior squares of elementary groups. In~\S\ref{subsec:general}--\S\ref{subsec:levelcomputation} we develop a technique for an arbitrary general exterior power. Finally, a level reduction for exterior powers is proved in~\S\ref{subsec:normgeneral}.

\textbf{Acknowledgment.} We would like to express our sincere gratitude to our scientific adviser Nikolai Vavilov for formulating the problem and for a constant support, without which this paper would never have been written. The authors are grateful to Alexei Stepanov for carefully reading our original manuscript and for numerous remarks and corrections. Also, we would like to thank an anonymous referee for bringing our attention to the paper \cite{GariGura2015}.

\section{Main results}\label{sec:main_results}

Fundamental representations of the general linear group $\GL_{n}$, as well as of the special linear group $\SL_{n}$, are the ones with the highest weights $\varpi_{m} = \underbrace{(1, \dots, 1)}_{m}$ for $m = 1, \dots, n$. The representation with the highest weight $\varpi_{n}$ degenerates for the group $\SL_{n}$. The explicit description of these representations uses exterior powers of the standard representation.

In detail, for a commutative ring $R$ by $\bw{m} R^{n}$ we denote an $m$-th exterior power of the free module $R^n$. We consider the following natural transformation, an \textit{exterior power}, 
$$\bw{m}: \GL_{n} \rightarrow \GL_{\binom{n}{m}}$$
which extends the action of the group $\GL_{n}(R)$ from $R^n$ to $\bw{m} R^{n}$. 

An elementary group $\E_n( R)$ is a subgroup of the group of points $\GL_{n}(R)$, so its exterior power $\bw{m}\E_n(R)$ is a well defined subgroup of the group of points $\bw{m}\GL_{n} (R)$. A more user--friendly description of the elementary group $\bw{m}\E_n(R)$ will be presented in Subsections~\ref{subsec:square} and~\ref{subsec:general}.

Let $H$ be an arbitrary overgroup of an elementary group $\bw{m}\E_n(R)$:
$$\bw{m}\E_n(R) \leq H \leq \GL_{\binom{n}{m}}(R).$$ 
For any unequal weights $I, J \in \bw{m}[n]$, which are indices for matrix entries of elements from $\GL_{\binom{n}{m}}(R)$, by $A_{I,J}$ we denote the following set 
$$A_{I,J}:=\{\xi \in R \;|\; t_{I,J}(\xi) \in H \} \subseteq R.$$

It turns out these sets are ideals that coincide for any pair of unequal weights $I\neq J$.

\begin{proposition}\label{prop:InversInclus}
Sets $A_{I,J}$ coincide for $n \geq 3m$. 
\end{proposition}

In the case $\frac{n}{3} \leq m \leq n$ description of overgroups cannot be done by a parametrization only by a single ideal. Moreover, as it could be seen from further calculations we need up to $m$ ideals in some cases for a complete parametrization of overgroups. There are a lot of nontrivial relationships between the ideals. So even the notion of a relative elementary group is far more complicated and depends on a Chevalley group (for instance, see~\cite{AnaVavSinI}), let alone formulations of the Main Theorems. The authors work in this direction and hope this problem would be solved in the near future. In a general case this [partially ordered] set of ideals forms \textit{a net of ideals} (due to Zenon Borevich; for a definition see \cite{BVnets}, for further progress in the direction of subgroup classification see \cite{BVdet, Borevich1984}).

Back to the case $n \geq 3m$, the set $A:= A_{I,J}$ is called a \textit{level of an overgroup} $H$. The description of overgroups goes as follows.

\begin{restatable}[Level computation]{theorem}{LevelForm}
\label{thm:LevelForm}
Let $R$ be a commutative ring and $n$, $m$ be natural numbers with the constraint $n\geq 3m$. For an arbitrary overgroup $H$ of the group $\bw{m}\E_n(R)$ there exists a unique maximal ideal $A$ of the ring $R$ such that 
$$\bw{m}\E_n(R)\cdot\E_N(R,A)\leq H.$$
Namely, if a transvection $t_{I,J}(\xi)$ belongs to the group $H$, then $\xi\in A$.
\end{restatable}

The left--hand side subgroup is denoted by $\E\bw{m}\E_n(R,A)$. We note that this group is perfect (Lemma~\ref{lem:PerfectForm}). Motivated by the expected relations $(*)$, we present an alternative description of the normalizer $N_{\GL_{N}(R)}\bigl(\E\bw{m}\E_n(R,A)\bigr)$. 

For this we introduce the canonical projection $\rho_A: R\longrightarrow R/A$ mapping $\lambda\in R$ to $\bar{\lambda}=\lambda+A\in R/A$. Applying the projection to all entries of a matrix, we get the reduction homomorphism 
$$\begin{array}{rcl}
\rho_{A}:\GL_{n}(R)&\longrightarrow& \GL_{n}(R/A)\\
a &\mapsto& \overline{a}=(\overline{a}_{i,j})
\end{array}$$

Eventually, we have the following explicit \textit{congruence} description.

\begin{restatable}[Level reduction]{theorem}{LevelReduction}
\label{thm:LevelReduction}
Let $n\geq 3m$.
For any ideal $A\trianglelefteq R$, we have
$$N_{\GL_{N}(R)}\bigl(\E\bw{m}\E_n(R,A)\bigr)=\rho_{A}^{-1}\left(\bw{m}\GL_{n}(R/A)\right).$$
\end{restatable}

\section{Principal notation}\label{sec:princnot}

Our notation for the most part is fairly standard in Chevalley group theory. We recall all necessary notion below for the purpose of self-containment.

First, let $G$ be a group. By a commutator of two elements we always mean \textit{the left-normed} commutator $[x,y]=xyx^{-1}y^{-1}$, where $x,y\in G$. Multiple commutators are also left-normed; in particular, $[x,y,z]=[[x,y],z]$. By ${}^xy=xyx^{-1}$ we denote \textit{the left conjugates} of $y$ by $x$. Similarly, by $y^x=x^{-1}yx$ we denote \textit{the right conjugates} of $y$ by $x$. In the sequel, we will use the Hall–Witt identity:
$$[x,y^{-1},z^{-1}]^x\cdot[z,x^{-1},y^{-1}]^z\cdot[y,z^{-1},x^{-1}]^y=e.$$

For a subset $X\subseteq G$, we denote by $\langle X\rangle$ a subgroup it generates. The notation $H\leq G$ means that $H$ is a subgroup in $G$, while the notation $H\trianglelefteq G$ means that $H$ is a normal subgroup in $G$. For $H\leq G$, we denote by $\langle X\rangle^H$ the smallest subgroup in $G$ containing $X$ and normalized by $H$. For two groups $F,H\leq G$, we denote by $[F,H]$ their mutual commutator: $[F,H]=\langle [f,g] \text{ for } f\in F, h\in H\rangle.$

Also we need some elementary ring theory notation. Let $R$ be an associative ring with $1$. By default, it is assumed to be commutative. By an ideal $I$ of a ring $R$ we understand \textit{the two-sided ideal} and this is denoted by $I\trianglelefteq R$. As usual, $R^{*}$ denotes a multiplicative group of a ring $R$. A multiplicative group of matrices over a ring $R$ is called a general linear group and is denoted by $\GL_n(R)=\M_n(R)^{*}$. A special linear group $\SL_{n}(R)$ is a subgroup of $\GL_n(R)$ consisting of matrices of determinant $1$. By $a_{i,j}$ we denote an entry of a matrix $a$ at the position $(i,j)$, where $1\leq i,j\leq n$. Further, $e$ denotes the identity matrix and $e_{i,j}$ denotes the standard matrix unit, i.\,e., the matrix that has $1$ at the position $(i,j)$ and zeros elsewhere. For entries of the inverse matrix we will use the standard notation $a_{i,j}':=(a^{-1})_{i,j}$.

By $t_{i,j}(\xi)$ we denote an elementary transvection, i.\,e., a matrix of the form $t_{i,j}(\xi)=e+\xi e_{i,j}$, $1\leq i\neq j\leq n$, $\xi\in R$. Hereinafter, we use (without any references) standard relations~\cite{StepVavDecomp} among elementary transvections such as
\begin{enumerate}
\item additivity:
$$t_{i,j}(\xi)t_{i,j}(\zeta)=t_{i,j}(\xi+\zeta).$$
\item the Chevalley commutator formula:
$$[t_{i,j}(\xi),t_{h,k}(\zeta)]=
\begin{cases}
e,& \text{ if } j\neq h, i\neq k,\\
t_{i,k}(\xi\zeta),& \text{ if } j=h, i\neq k,\\
t_{h,j}(-\zeta\xi),& \text{ if } j\neq h, i=k.
\end{cases}$$
\end{enumerate}

A subgroup $\E_n(R)\leq \GL_n(R)$ generated by all elementary transvections is called an \textit{$($absolute$)$ elementary group}:
$$\E_n(R)=\langle t_{i,j}(\xi), 1\leq i\neq j\leq n, \xi\in R\rangle.$$
Now define a normal subgroup of $\E_n(R)$, which plays a crucial role in calculating the level of intermediate subgroups. Let $I$ be an ideal in $R$. Consider a subgroup $\E_n(R,I)$ generated by all elementary transvections of level $I$, i.\,e., $\E_n(R,I)$ is a normal closure of $\E_n(I)$ in $\E_n(R)$. This group is called an \textit{$($relative$)$ elementary group of level $I$}:
$$\E_n(R,I)=\langle t_{i,j}(\xi), 1\leq i\neq j\leq n, \xi\in I\rangle^{\E_n(R)}.$$
It is well known (due to Andrei Suslin~\cite{SuslinSerreConj}) that the elementary group is normal in the general linear group $\GL_{n}(R)$ for $n \geq 3$. The normality is crucial for further considerations, so hereafter we suppose that $n \geq 3$. Furthermore, the relative elementary group $\E_n(R,I)$ is normal in $\GL_{n}(R)$ if $n\geq 3$. This fact, first proved in \cite{SuslinSerreConj}, is cited as Suslin’s theorem. Moreover, if $n\geq 3$, then the group $\E_n(R,I)$ is generated by transvections of the form $z_{i,j}(\xi,\zeta)=t_{j,i}(\zeta)t_{i,j}(\xi)t_{j,i}(-\zeta)$, $1\leq i\neq j\leq n$, $\xi\in I$, $\zeta\in R$. This fact was proved by Leonid~Vaserstein and Andrey~Suslin~\cite{VasSusSerre} and, in the context of Chevalley groups, by Jacques~Tits~\cite{TitsCongruence}.

By $[n]$ we denote the set $\{1,2,\dots, n\}$ and by $\bw{m}[n]$ we denote an exterior power of the set $[n]$. Elements of $\bw{m}[n]$ are ordered subsets $I\subseteq [n]$ of cardinality $m$ without repeating entries:
$$\bw{m}[n] = \{ (i_{1}, i_{2}, \dots , i_{m})\; |\; i_{j} \in [n], i_{j} \neq i_{l} \}.$$
We use the lexicographic order on $\bw{m}[n]$ by default: $12\dots (m-1)m < 12\dots (m-1)(m+1) < \dots$

Usually, we write an index $I=\{i_j\}_{j=1}^m$ in the ascending order, $i_1<i_2<\dots<i_m$. Sign $\sign(I)$ of the index $I = (i_1, \dots, i_m)$ equals the sign of the permutation mapping $(i_1, \dots, i_m)$ to the same set in the ascending order. For example, $\sign(1234) = \sign(1342) = +1$, but $\sign(1324) = \sign(4123) = -1$.

Finally, let $n\geq 3$ and $m\leq n$. \textit{By $N$ we denote the binomial coefficient $\binom{n}{m}$}. In the sequel, we denote an elementary transvection in $\E_N(R)$ by $t_{I,J}(\xi)$ for $I,J \in \bw{m}[n]$ and $\xi \in R$. For instance, the transvection $t_{12,13}(\xi)$ equals the matrix with $1$'s on the diagonal and $\xi$ in the position $(12,13)$.

\section{Proofs \& Computations}\label{sec:proofs}

We consider a case of an exterior square of a group scheme $\GL_{n}$ at first. We have two reasons for this way of presentation. Firstly, proofs of statements in a general case belong to the type of technically overloaded statements. At the same time, simpler proofs in the basic case present all ideas necessary for a general case. In particular, for $n=4$ Nikolai Vavilov and Victor Petrov completed the standard description of overgroups\footnote{The restriction of the exterior square map $\bw{2}: \GL_{4}(R) \longrightarrow \GL_{6}(R)$ to the group $\E_4(R)$ is an isomorphism onto the elementary orthogonal group $\EO_6(R)$~\cite{VP-EOeven}.}. Secondly, for exterior squares there are several important results that cannot be obtained for the exterior cube or other powers, see~\cite{LubNek18,LubReverse2,LubReversemArx}. For example, in~\cite{LubReverse2} the author construct a transvection $T\in \bw{2}\E_n(R)$ such that it stabilizes an arbitrary column of a matrix $g$ in $\GL_{\binom{n}{2}}(R)$. And there are no such transvections for other exterior powers.

\subsection{Exterior square of elementary groups}\label{subsec:square}

Let $R$ be a commutative ring with $1$, $n$ be a natural number greater than $3$, and $R^n$ be a right free $R$-module with the standard basis $\{e_1,\dots,e_n\}$. By $\bw{2}R^n$ we denote a universal object in the category of alternating bilinear maps from $R^n$ to $R$-modules. Concretely, take a free module of rank $N=\binom{n}{2}$ with the basis $e_i\wedge e_j$, $1~\leq~i~\neq~j~\leq~n$. The elements $e_i\wedge e_j$ for arbitrary $1\leq i,j\leq n$ are defined by the relation $e_{i} \wedge e_{j} = -e_{j} \wedge e_{i}$.

An action of the group $\GL_{n}(R)$ on the module $\bw{2}R^n$ is diagonal:
$$\bw{2}(g)(e_i\wedge e_j):=(ge_i)\wedge(ge_j) \text{ for any } g\in \GL_{n}(R) \text{ and } 1\leq i\neq j\leq n.$$
In the basis $\{e_I, I\in\bw{2}[n]\}$ of the module $\bw{2}R^n$ a matrix $\bw{2}(g)$ consists of second order minors of the matrix $g$ with lexicographically ordered columns and rows:
$$\left(\bw{2}(g)\right)_{I,J}=\left(\bw{2}(g)\right)_{(i_1,i_2),(j_1,j_2)}=M_{i_1,i_2}^{j_1,j_2}(x) = g_{i_1, j_1}\cdot g_{i_2, j_2} - g_{i_1, j_2}\cdot g_{i_2, j_1}.$$

By the Cauchy--Binet theorem the map $\pi\colon\GL_{n}(R)\longrightarrow \GL_{N}(R)$, $x\mapsto \bw{2}(x)$ is a homomorphism. Thus the map $\pi$ is a representation of the group $\GL_n(R)$, called the \textit{bivector representation} or the \textit{second fundamental representation} (\textit{the representation with the highest weight} $\varpi_2$). The image of the latter action is called the exterior square of the group $\GL_n(R)$. $\E_n(R)$ is a subgroup of $\GL_n(R)$, therefore the exterior square of the elementary group is well defined. The following lemma is a corollary of Suslin's theorem.
\begin{lemma}\label{SuslinFor2}
The image of an elementary group is normal in the image of a general linear group under the exterior square homomorphism:
$$\bw{2}\bigl(\E_n(R)\bigr)\trianglelefteq \bw{2}\bigl(\GL_{n}(R)\bigr).$$
\end{lemma}

Note that $\bw{2}\bigl(\GL_{n}(R)\bigr)$ does not equal $\bw{2}\GL_n(R)$ for arbitrary rings. For detail see the extended description in \S\ref{subsec:general}.

Let us consider a structure of the group $\bw{2}\E_n(R)$ in detail. The following proposition can be extracted from the very definition of $\bw{2}\big(\GL_{n}(R)\big)$.
\begin{proposition}\label{prop:ImageOfTransvFor2}
Let $t_{i,j}(\xi)$ be an elementary transvection. For $n\geq 3$, $\bw{2}t_{i,j}(\xi)$ can be presented as the following product:
\begin{equation}
\bw{2}t_{i,j}(\xi)=\prod\limits_{k=1}^{i-1} t_{ki,kj}(\xi)\,\cdot\prod\limits_{l=i+1}^{j-1}t_{il,lj}(-\xi)\,\cdot\prod\limits_{m=j+1}^n t_{im,jm}(\xi)
\label{eq:def2}
\end{equation}
for any $1\leq i<j\leq n$.
\end{proposition}

\begin{remark}
For $i>j$ a similar equality holds:
$$\bw{2}t_{i,j}(\xi)=\prod\limits_{k=1}^{j-1} t_{ki,kj}(\xi)\,\cdot\prod\limits_{l=j+1}^{i-1}t_{li,jl}(-\xi)\,\cdot\prod\limits_{m=i+1}^n t_{im,jm}(\xi).$$
\end{remark}

Likewise, one can get an explicit form of torus elements $h_{\varpi_2}(\xi)$ of the group $\bw{2}\GL_{n}(R)$.
\begin{proposition}\label{prop:ImageOfDiag2}
Let $d_i(\xi)=e+(\xi-1)e_{i,i}$ be a torus generator, $1\leq i\leq n$. Then the exterior square of $d_i(\xi)$ equals a diagonal matrix, with diagonal entries 1 everywhere except in $n-1$ positions:
\begin{equation}
\bw{2}\bigl(d_i(\xi)\bigr)_{I,I}=
\begin{cases}
\xi,& \text{ if } i\in I,\\
1,& \text{ otherwise}.
\end{cases}
\label{eq:diag2}
\end{equation}
\end{proposition}

It follows from the propositions that $\bw{2}t_{i,j}(\xi)\in \E^{n-2}(N,R)$, where a set $\E^M(N,R)$ consists of products of $M$ or less elementary transvections, e.\,g., $\bw{2}t_{1,3}(\xi)=t_{12,23}(-\xi)t_{14,24}(\xi)t_{15,25}(\xi)\in\bw{2}E_5(R)$.
 
Let $H$ be an overgroup of the exterior square of the elementary group $\bw{2}\E_n(R)$:
$$\bw{2}\E_n(R)\leq H \leq \GL_N(R).$$
We consider two indices $I, J \in \bw{2}[n]$. By $A_{I,J}$ we denote the set 
$$A_{I,J}:=\{\xi \in R \;|\; t_{I,J}(\xi) \in H \} \subseteq R.$$
By definition diagonal sets $A_{I,I}$ equals whole ring $R$ for any index $I$. In the rest of the section, we prove that these sets are ideals, i.\,.e., $A_{I,J}$ form a net of ideals. Moreover, we will get $D$-net in terms of Zenon Borevich~\cite{BVnets} by the latter statement.

Let $t_{I,J}(\xi)$ be an elementary transvection. We define a \textit{height} of $t_{I,J}(\xi)$ (generally, of the  pair $(I,J)$) as a cardinality of the intersection $I\cap J$:
$$\height(t_{I,J}(\xi))=\height(I,J)=|I\cap J|.$$
This combinatorial characteristic of transvections is useful in commutator calculations. 

The height splits up all sets $A_{I,J}$ into two classes: the one with $\height(I,J)=0$ and the other with $\height(I,J)=1$. In fact, these classes are equal for $n\geq 6$. The set $A := A_{I,J} $ is called  a \textit{level} of an overgroup $H$. Note that for $n=4$ the level is unique, that follows from~\cite{VP-EOeven}.

\begin{lemma}\label{lem:IdealFor2}
If $n\geq 6$, then every set $A_{I,J}$ is an ideal of the ring $R$. Moreover, for any $I\neq J$ and $K\neq L$ the ideals $A_{I,J}$ and $A_{K,L}$ coincide.
\end{lemma}
\begin{proof} A complete proof is presented in Section~\ref{subsec:levelcomputation}, Proposition~\ref{prop:LevelsEqualGeneral}. Here we sketch calculations in the case $(n,m) = (4,2)$ exclusively. These calculations present the general idea in a transparent way.
\begin{enumerate}
\item Firstly, take any $\xi\in A_{12,34}$, i.\,e., $t_{12,34}(\xi)\in H$. Then 
$$[t_{12,34}(\xi),\bw{2}t_{4,2}(\zeta)]=t_{14,23}(-\xi\zeta^2)t_{14,34}(-\zeta\xi)t_{12,23}(-\xi\zeta)\in H.$$
It remains to provide this calculation with $-\zeta$ and to product two right-hand sides; then we obtain $t_{14,23}(-2\xi\zeta^2)\in H$. By the condition $2 \in R^{*}$, this means that $A_{12, 34} \subseteq A_{14, 23}$.  It follows that
$$A_{I,J} \subseteq A_{K,L}\text{ for } I\cup J = K \cup L = \{ 1234 \}.$$
\item Secondly, take any $\xi\in A_{12,34}$, then $[t_{12,34}(\xi),\bw{2}t_{4,5}(\zeta)]=t_{12,35}(\xi\zeta)$. Consequently, 
$$A_{I,J} \subseteq A_{K,L}\text{ for } \height(I, J) = \height(K, L) = 0.$$
\item Thirdly, let $\xi\in A_{12,13}$, then $[t_{12,13}(\xi),\bw{2}t_{1,4}(\zeta)]=t_{12,34}(-\xi\zeta)\in H$. Consider two commutators of the latter transvection with $\bw{2}t_{4,1}(\zeta_1)$ and $\bw{2}t_{4,1}(-\zeta_1)$ respectively. We obtain that $t_{24,13}(\zeta_1^2\xi\zeta)\in H$ and also $t_{12,13}(-\xi\zeta\zeta_1)t_{24,34}(\zeta_1\xi\zeta)\in H$. Hence $t_{24,34}(\zeta_1\xi\zeta)\in H$. This means that 
$$A_{I,J} \subseteq A_{K,L}\text{ for any } \height(I, J) = \height(K, L) = 1.$$
\item Now, take any $\xi\in A_{12,23}$, then $[t_{12,23}(\xi),\bw{2}t_{4,2}(\zeta)]=t_{14,23}(-\zeta\xi)$. Thus
$$A_{I,J} \subseteq A_{K,L}\text{ for } \height(I, J)= 1, \height(K, L) = 0.$$
\item Finally, let $\xi\in A_{12,34}$. As in (1), consider the commutator $t_{12,34}(\xi)$ with $\bw{2}t_{4,2}(\zeta)$. We obtain $t_{14,23}(-2\xi\zeta^2)\in H$ and $t_{14,34}(-\zeta\xi)t_{12,23}(-\xi\zeta)\in H$. By the same argument we can provide these calculations with the transvection $t_{45,16}(\xi)$ and $\bw{2}t_{6,4}(\zeta_1)$. We get that $t_{56,14}(-2\zeta_1^2\xi)\in H$ and $t_{45,14}(\xi\zeta_1)t_{56,16}(\zeta_1\xi)\in H$. To finish the proof it remains to commutate latter two products. Then $t_{45,34}(-\xi^2\zeta_1\zeta)\in H$, or  
$$A_{I,J} \subseteq A_{K,L}\text{ for }\height(I, J)= 0, \height(K, L) = 1.$$
\end{enumerate}
\end{proof}

The following lemma is crucial for the rest. It gives an alternative description of the relative elementary group.
\begin{lemma}\label{lem:AlterRelatFor2}
Let $n\geq 6$. For any ideal $A\trianglelefteq R$, we have
$$\E_N(A)^{\bw{2}\E_n(R)}=\E_N(R,A),$$ where by definition $\E_N(R,A)=\E_N(A)^{\E_N(R)}.$
\end{lemma}
\begin{proof}
The inclusion $\leq$ is trivial. By Vaserstein--Suslin's lemma~\cite{VasSusSerre}, the group $\E_N(R,A)$ is generated by elements of the form
$$z_{ij,hk}(\xi,\zeta) = z_{I,J}(\xi,\zeta) = t_{J,I}(\zeta)\,t_{I,J}(\xi)\,t_{J,I}(-\zeta), \;\xi \in A, \zeta \in R.$$
Hence to prove the reverse inclusion, it sufficient to check the
matrix $z_{ij,hk}(\xi,\zeta)$ to belong to $F:=\E_N(A)^{\bw{2}\E_n(R)}$ for any $\xi \in A$, $\zeta \in R$. Let us consider two cases:
\begin{itemize}
\item Suppose that there exists one pair of the same indices. Without loss of generality, we can assume that $i=k$. Then this inclusion is obvious:
$$z_{ij,hi}(\xi,\zeta)={}^{t_{hi,ij}(\zeta)}t_{ij,hi}(\xi)={}^{\bw{2}t_{h,j}(\zeta)}t_{ij,hi}(\xi)\in F.$$
\item Thus, we are left with the inclusion $z_{ij,hk}(\xi,\zeta)\in F$ with different indices $i,\,j\,,h\,,k$. Firstly, we express
$t_{ij,hk}(\xi)$ as a commutator of elementary transvections:
$$z_{ij,hk}(\xi,\zeta)={}^{t_{hk,ij}(\zeta)}t_{ij,hk}(\xi)={}^{t_{hk,ij}(\zeta)}[t_{ij,jh}(\xi),t_{jh,hk}(1)].$$
Conjugating arguments of the commutator by $t_{hk,ij}(\zeta)$, we get
$$z_{ij,hk}(\xi,\zeta)=[t_{ij,jh}(\xi)t_{hk,jh}(\zeta \xi),t_{jh,ij}(-\zeta)t_{jh,hk}(1)] =:[ab,cd].$$
Next, we decompose the right-hand side with a help of the formula 
$$[ab,cd] ={}^a[b,c]\cdot {}^{ac}[b,d]\cdot[a,c]\cdot {}^c[a,d],$$
and observe the exponent $a$ belongs to $\E_N(A)$, so can be ignored. Now a direct calculation, based upon the Chevalley commutator formula, shows that
\begin{align*}
[b,c]&=[t_{hk,jh}(\zeta \xi),t_{jh,ij}(-\zeta)]=t_{hk,ij}(-\zeta^2 \xi) \in \E_N(A);\\
^c[b,d]&={}^{t_{jh,ij}(-\zeta)}[t_{hk,jh}(\zeta \xi),t_{jh,hk}(1)]=\\
&=t_{hk,ik}(-\xi\zeta^2(1+\xi\zeta))t_{jh,ik}(-\xi\zeta^2)\cdot \,{}^{\bw{2}t_{h,i}(\zeta)}[t_{hk,jh}(\xi\zeta),\bw{2}t_{j,k}(-1)];\\
[a,c]&=[t_{ij,jh}(\xi),t_{jh,ij}(-\zeta)]=[t_{ij,jh}(\xi),\bw{2}t_{h,i}(-\zeta)];\\
^c[a,d]&={}^{t_{jh,ij}(-\zeta)}[t_{ij,jh}(\xi),t_{jh,hk}(1)]=\\
&=t_{jh,ik}(\xi\zeta^2)t_{ij,ik}(-\xi\zeta)\cdot\,{}^{\bw{2}t_{h,i}(\zeta)}[t_{ij,jh}(\xi),\bw{2}t_{j,k}(-1)],
\end{align*}
where all factors on the right-hand side belong to $F$.
\end{itemize}
\end{proof}

\begin{remark}
The attentive reader can remark these calculations to be almost completely coincide with the calculations for the orthogonal and symplectic cases~\cite{VP-EOeven,VP-EOodd,VP-Ep}. In the special case $(n,m) = (4,2)$ calculations are the same due to the isomorphism $\bw{2}\E_4(R) \cong \EO_6(R)$. Amazingly this argument proves a similar proposition in the case of general exterior power (see Section~\ref{subsec:levelcomputation}, Lemma~\ref{lem:AlterRelatForm}).
\end{remark}

\begin{corollary}\label{cor:CorolOfL2}
Let $A$ be an arbitrary ideal of $R$. Then
$$\bw{2}\E_n(R)\cdot\E_N(R,A)=\bw{2}\E_n(R)\cdot\E_N(A).$$
\end{corollary}

Summarizing above two lemmas, we get the main result of the paper for bivectors.
\begin{theorem}[Level Computation]\label{thm:LevelFor2}
Let $n\geq 6$ and let $H$ be a subgroup in $\GL_{N}(R)$ containing $\bw{2}\E_n(R)$. Then there exists a unique maximal ideal $A\trianglelefteq R$ such that
$$\bw{2}\E_n(R)\cdot\E_N(R,A)\leq H.$$
Namely, if $t_{I,J}(\xi)\in H$ for some $I$ and $J$, then $\xi\in A$.
\end{theorem}

Lemma~\ref{lem:AlterRelatFor2} asserts precisely $\bw{2}\E_n(R)\cdot\E_N(R,A)$ to be generated as a subgroup by transvections $\bw{2}t_{i,j}(\zeta)$, $\zeta\in R$, and by elementary transvections $t_{ij,hk}(\xi)$, $\xi\in A$ of level $A$. As usual, we assume that $n\geq 6$ and $2\in R^{*}$.

We formulate a perfectness of the lower bound subgroup from the latter Theorem. The proof follows from Lemma \ref{lem:PerfectForm}.
\begin{lemma}\label{lem:PerfectFor2}
Let $n\geq 6$. The group $\bw{2}\E_n(R)\cdot\E_N(R,A)$ is perfect for any ideal $A\trianglelefteq R$.
\end{lemma}

\subsection{Exterior powers of elementary groups}\label{subsec:general}

In this section, we lift the previous statements from the level of an exterior square to the case of an arbitrary exterior power functor.

Let us define an $m$-th exterior power of an $R$-module $R^n$ as follows. A basis of this module consists of exterior products $e_{i_1}\wedge\dots\wedge e_{i_m}$, where $1\leq i_1<\dots<i_m\leq n$. Products $e_{i_1}\wedge\dots\wedge e_{i_m}$ are defined for any set $i_1,\dots,i_m$ as $e_{\sigma(i_{1})}\wedge\ldots\wedge e_{\sigma(i_{m})} = \sign(\sigma)\, e_{i_1} \wedge \ldots \wedge e_{i_{m}}$ for any permutation $\sigma$ in the permutation group $S_{m}$. We denote the $m$-th exterior power of $R^n$ by $\bw{m}R^n$. 

For every $m$ the group $\GL_{n}(R)$ acts diagonally on the module $\bw{m}R^n$. Namely, an action of a matrix $g\in\GL_n(R)$ on decomposable $m$-vectors is set according to the rule
$$\bw{m}(g)(e_{i_1}\wedge\dots\wedge e_{i_m}):=(ge_{i_1})\wedge\dots\wedge (ge_{i_m})$$
for every $e_{i_1},\dots,e_{i_m}\in R^n$. In the basis $e_I, I\in\bw{m}[n]$ a matrix $\bw{m}(g)$ consists of $m$-order minors of the matrix $g$ with lexicographically ordered columns and rows:
$$\left(\bw{m}(g)\right)_{I,J}=\left(\bw{m}(g)\right)_{(i_1,\dots,i_m),(j_1,\dots,j_m)}=M_{i_1,\dots,i_m}^{j_1,\dots,j_m}(g).$$

By the Cauchy--Binet theorem the map $\pi\colon\GL_{n}(R)\longrightarrow \GL_{N}(R)$, $x\mapsto \bw{m}(x)$ is homomorphism. Thus, the map $\pi$ is a representation of the group $\GL_n(R)$ called the $m$-th \textit{vector representation} or the $m$-th \textit{fundamental representation} (\textit{the representation with the highest weight} $\varpi_m$). The image of the latter action is called the $m$-th exterior power of the group $\GL_n(R)$. $\E_n(R)$ is a subgroup of $\GL_n(R)$, therefore the exterior power of the elementary group is well defined.

We cannot but emphasize the difference for arbitrary rings between the groups\footnote{The same strict inclusions are still true with changing $\GL$ to $\SL$.}
$$\bw{m} \bigl(\GL_{n}(R)\bigr) < \bw{m}\GL_{n}(R) < \GL_{\binom{n}{m}}(R).$$
The first group is a set-theoretic image of the [abstract] group $\GL_{n}(R)$ under the Cauchy--Binet homomorphism $\bw{m}:\GL_{n}(R)\longrightarrow \GL_{\binom{n}{m}}\left(R\right)$, while the second one is a group of $R$-points of the \textbf{categorical} image of the group scheme $\GL_{n}$ under the natural transformation corresponding to the Cauchy--Binet homomorphism. Since the epimorphism of algebraic groups on points is not surjective in this situation, we see that $\bw{m}\GL_n(R)$ is strictly larger than $\bw{m}\bigl(\GL_{n}(R)\bigr)$. In fact, elements of $\bw{m}\GL_n(R)$ are still images of matrices, but coefficients are not from the ring itself, but from its extensions. This means that for any commutative ring $R$ elements $\widetilde{g}\in\bw{m}\GL_n(R)$ can be represent in the form $\widetilde{g}=\bw{m}g$, $g\in\GL_n(S)$, where $S$ is an extension of the ring $R$. We refer the reader to~\cite{VavPere} for more precise results about the difference between these groups.

As in Section~\ref{subsec:square}, $\bw{m}\E_n(R)$ is a normal subgroup of $\bw{m}(\GL_{n}(R))$ by Suslin's lemma. Moreover, $\bw{m}\E_n(R)$ is normal in $\bw{m}\GL_n(R)$. This fact follows from~\cite[Theorem~1]{PetrovStavrovaIsotropic}.
\begin{theorem}\label{thm:normalityPS}
Let $R$ be a commutative ring, $n\geq 3$, then
$\bw{m}\E_n(R)\trianglelefteq\bw{m}\GL_n(R)$.
\end{theorem}

For further computations we calculate an exterior power of an elementary transvection in the following proposition. The proof is straightforward by the very definition of the [classical] Binet--Cauchy homomorphism.

\begin{proposition}\label{prop:ImageOfTransvForm}
Let $t_{i,j}(\xi)$ be an elementary transvection in $\E_n(R)$, $n\geq 3$. Then $\bw{m}t_{i,j}(\xi)$ equals
\begin{equation}
\bw{m}t_{i,j}(\xi)=\prod\limits_{L\,\in\,\bw{m-1}\,[n\setminus \{i,j\}]} t_{L\cup i,L\cup j}(\sign(L, i)\sign(L, j)\xi)
\label{eq:m}
\end{equation}
for any $1\leq i\neq j\leq n$.
\end{proposition}

Similarly, one can get an explicit form of torus elements $h_{\varpi_m}(\xi)$ of the group $\bw{m}\GL_{n}(R)$.
\begin{proposition}\label{prop:ImageOfDiag}
Let $d_i(\xi)=e+(\xi-1)e_{i,i}$ be a torus generator, $1\leq i\leq n$. Then the exterior power of $d_i(\xi)$ equals a diagonal matrix, with diagonal entries 1 everywhere except in $n-1$ positions:
\begin{equation}
\bw{m}(d_i(\xi))_{I,I}=
\begin{cases}
\xi,& \text{ if } i\in I,\\
1,& \text{ otherwise}.
\end{cases}
\label{eq:diagm}
\end{equation}
\end{proposition}

As an example, consider $\bw{3}t_{1,3}(\xi)=t_{124,234}(-\xi)t_{125,235}(-\xi)t_{145,345}(\xi)\in\bw{3}\E_5(R)$ and $\bw{4}d_2(\xi)=\opn{diag}(\xi,\xi,\xi,1,\xi)\in\bw{4}\E_5(R)$. It follows from the propositions $\bw{m}t_{i,j}(\xi)\in \E^{\binom{n-2}{m-1}}(N,R)$, where by definition every element of the set $\E^M(N,R)$ is a product of $M$ or less elementary transvections. In other words, a residue of a transvection $\opn{res}(\bw{m}t_{i,j}(\xi))$ equals the binomial coefficient $\binom{n-2}{m-1}$. Recall that a residue $\opn{res}(g)$ of a transformation $g$ is called the rank of $g-e$. Finally, there is a simple connection between the determinant of a matrix $g\in\GL_n(R)$ and the determinant of $\bw{m}g\in\bw{m}\GL_n(R)$, see~\cite[Proof of Theorem~4]{WaterhousePGL}:
$$\det\bw{m}g=\bigl(\det(g)\bigr)^{\binom{n}{m}\cdot\frac{m}{n}}=\bigl(\det(g)\bigr)^{\binom{n-1}{m-1}}.$$

\subsection{Elementary calculations technique}\label{subsec:levelgeneral}

For an arbitrary exterior power calculations with elementary transvections are huge. In this section, we organize all possible calculations of a commutator of an elementary transvection with an exterior transvection.

\begin{proposition}\label{prop:TypesOfComm}
Up to the action of the permutation group there exist three types of commutators with a fixed transvection $t_{I,J}(\xi) \in \E_N(R)$:
\begin{enumerate}
\item $[t_{I,J}(\xi), \bw{m}t_{j,i}(\zeta)]=1$ if both $i\not\in I$ and $j \not \in J$ hold;
\item $[t_{I,J}(\xi), \bw{m}t_{j,i}(\zeta)] = t_{\tilde{I}, \tilde{J}}(\pm \zeta\xi)$ if either $i\in I$ or $j \in J$. And then $\tilde{I} = I\backslash i \cup j$ or $\tilde{J} = J\backslash j \cup i$ respectively;
\item If both $i\in I$ and $j \in J$ hold, then we have the equality:
$$[t_{I,J}(\xi), \bw{m}t_{j,i}(\zeta)] = t_{\tilde{I}, J}(\pm \zeta\xi)\cdot t_{I, \tilde{J}}(\pm \zeta\xi)\cdot t_{\tilde{I}, \tilde{J}}(\pm\zeta^2\xi).$$
\end{enumerate}
\end{proposition}
Note that the latter case is true whenever $I\setminus i\neq J\setminus j$, otherwise we obtain $[t_{I,J}(\xi),t_{J,I}(\pm\zeta)]$. This commutator cannot be presented in a simpler form than the very definition.

The rule for commutator calculations from the latter proposition can be translated into the language of \textit{weight diagrams}:\\[7mm]
\textbf{Weight diagrams tutorial.}
\begin{enumerate}
\item Let $G(A_{n-1},\blank)$ be a Chevalley--Demazure group scheme, and let $(I,J)\in\bw{m}[n]^2$ be a pair of different weights for the $m$-th exterior power of $G(A_{n-1},\blank)$. Consider any unipotent $x_\alpha(\xi)$ for a root $\alpha$ of the root system $A_{n-1}$, i.\,e., $x_\alpha(\xi)$ equals an elementary transvection $\bw{m}t_{i,j}(\xi) \in \bw{m}\E_n( R)$;
\item By $\Ar(\alpha)$ denote all paths on the weight diagram\footnote{Recall that we consider the representation with the highest weight $\varpi_{m}$.} of this representation corresponding to the root $\alpha$;
\item Then there exist three different scenarios corresponding to the cases of Proposition~\ref{prop:TypesOfComm}:
\begin{itemize}
\item sets of the initial and the terminal vertices of paths from $\Ar(\alpha)$ do not contain the vertex $(I,J)$;
\item the vertex $(I,J)$ is initial or terminal for one path from $\Ar(\alpha)$;
\item the vertex $(I,J)$ is simultaneously initial and terminal for some path\footnote{From root systems geometry any vertex can be initial or terminal for at most one $\alpha$-path.} from $\Ar(\alpha)$.
\end{itemize}
\item Finally, let us consider a commutator of the transvection $t_{I,J}(\xi)$ and the element $\bw{m}t_{i,j}(\zeta)$. It equals a product of transvections. These transvections correspond to the paths from the previous step. Transvections' arguments are monomials in $\xi$ and $\zeta$. Namely, in the second case the argument equals $\pm\xi\zeta$; in the third case it equals $\pm\xi\zeta^2$.
\end{enumerate}

In Figure~\ref{Fig1}$(a)$ we present all three cases from step $(3)$ for $m=2$ and $\alpha = \alpha_{2}$:
\begin{itemize}
\item $(I,J) = (14,15)$, then $[t_{14, 15}(\xi), \bw{2}t_{2,3}(\zeta)] = 1$;
\item $(I,J) = (13,35)$, then $[t_{13, 35}(\xi), \bw{2}t_{2,3}(\zeta)] = t_{12, 35}(-\xi\zeta)$;
\item $(I,J) = (13,24)$, then $[t_{13, 24}(\xi), \bw{2}t_{2,3}(\zeta)] = t_{12, 24}(-\xi\zeta)t_{12,34}(\xi\zeta^2) t_{13,34}(\zeta\xi)$.
\end{itemize}

Similarly, for the case $m=3$ the elementary calculations can be seen directly from Figure~\ref{Fig1}$(b)$.

\[
\xymatrix @+1.0pc {
{\overset{12}{\bullet}}\ar@{-}[r] &{\overset{13}{\bullet}}\ar@{-}[r]\ar@{-}[d]\textbf{\ar@/_1pc/[l]_2}&{\overset{14}{\bullet}}\ar@{-}[r]\ar@{-}[d]&{\overset{15}{\bullet}}\ar@{-}[d]\\
&{\overset{23}{\bullet}}\ar@{-}[r]&{\overset{24}{\bullet}}\ar@{-}[r]\ar@{-}[d]&{\overset{25}{\bullet}}\ar@{-}[d] \\
&&{\overset{34}{\bullet}}\ar@{-}[r]\textbf{\ar@/_1pc/[u]_2}&{\overset{35}{\bullet}}\ar@{-}[d]\textbf{\ar@/_1pc/[u]_2}\\
&&&{\overset{45}{\bullet}}\\
&&(a)}
\hspace{5mm}
\xymatrix @-2.0pc {
{\overset{123}{\bullet}}\ar@{-}[rrrr] &&&&{\overset{124}{\bullet}}\ar@{-}[rrrr]\ar@{-}[ddrrr]&&&&\textbf{\ar@/_1pc/[llll]_4}{\overset{125}{\bullet}}\ar@{-}[rrrr]\ar@{-}[ddrrr]&&&&{\overset{126}{\bullet}}\ar@{-}[ddrrr]\\
\\
&&&&&&&{\overset{134}{\bullet}}\ar@{-}[rrrr]\ar@{-}[ddd] &&&&{\overset{135}{\bullet}}\ar@{-}[rrrr]\ar@{-}[ddrrr]\ar@{-}[ddd]\textbf{\ar@/_1pc/[llll]_4}&&&&{\overset{136}{\bullet}}\ar@{-}[ddrrr]\ar@{-}[ddd]\\
\\
&&&&&&&&&&&&&&{\overset{145}{\bullet}}\ar@{-}[rrrr]\ar@{-}[ddd] &&&&{\overset{146}{\bullet}}\ar@{-}[ddrrr]\ar@{-}[ddd]\\
&&&&&&&{\overset{234}{\bullet}}\ar@{-}[rrrr] &&&&{\overset{235}{\bullet}}\ar@{-}[rrrr]\ar@{-}[ddrrr]\textbf{\ar@/_1pc/[llll]_4}&&&&{\overset{236}{\bullet}}\ar@{-}[ddrrr]\\
&&&&&&&&&&&&&&&&&&&&&{\overset{156}{\bullet}}\ar@{-}[ddd]\textbf{\ar@/_1pc/[llluu]_4}\\
&&&&&&&&&&&&&&{\overset{245}{\bullet}}\ar@{-}[rrrr]\ar@{-}[ddd]&&&&{\overset{246}{\bullet}}\ar@{-}[ddrrr]\ar@{-}[ddd]\\
\\
&&&&&&&&&&&&&&&&&&&&&{\overset{256}{\bullet}}\ar@{-}[ddd]\textbf{\ar@/_1pc/[llluu]_4}\\
&&&&&&&&&&&&&&{\overset{345}{\bullet}}\ar@{-}[rrrr] &&&&{\overset{346}{\bullet}}\ar@{-}[ddrrr]\\
\\
&&&&&&&&&&&&&&&&&&&&&{\overset{356}{\bullet}}\ar@{-}[ddd]\textbf{\ar@/_1pc/[llluu]_4}\\
\\
\\
&&&&&&&&&&&&&&&&&&&&&{\overset{456}{\bullet}}\\
&&&&&&&&&&&&&&(b)}
\]
\captionof{figure}{Weight diagrams for $(a)$: $(A_4,\varpi_2)$, $\alpha = \alpha_{2}$ and $(b)$: $(A_5,\varpi_3)$, $\alpha = \alpha_{4}$}\label{Fig1}

\subsection{Level computation}\label{subsec:levelcomputation}
We generalize the notion of ideals $A_{I,J}$ to the case of the $m$-th exterior power. Let $H$ be an overgroup of the exterior power of the elementary group $\bw{m}\E_n(R)$:
$$\bw{m}\E_n(R)\leq H \leq \GL_N(R).$$
Let 
$$A_{I,J}:=\{\xi \in R \;|\; t_{I,J}(\xi) \in H \}$$
for any indices $I, J \in \bw{m}[n]$. As usual, diagonal sets $A_{I,I}$ equal the whole ring $R$ for any index $I\in\bw{m}[n]$. Thus, we will construct $D$-net of ideals of the ring $R$. Recall that the desired parametrization is given by an explicit juxtaposition for any overgroup $H$ its level, namely an ideal $A$ of the ring $R$. We compute this ideal $A$ in the present section.

We assume that $n \geq 2m$ due to the isomorphism $\bw{m}V^{*} \cong (\bw{\dim(V)-m}V)^{*}$ for an arbitrary free $R$-module $V$. The first step toward the level description is the following observation.

\begin{proposition}\label{prop:LevelsEqualGeneral}
If $|I \cap J|=|K\cap L|$, then sets $A_{I,J}$ and $A_{K,L}$ coincide. In fact, $A_{I,J}$ are ideals of $R$.
\end{proposition}
But first, we prove a weaker statement.
\begin{lemma} \label{lem:LevelsEqual0}
Let $I, J, K, L$ be different elements of the set $\bw{m}[n]$ such that $|I \cap J|=|K\cap L| = 0$. If $n \geq 2m$, then sets $A_{I,J}$ and $A_{K,L}$ coincide.
\end{lemma}
\begin{proof}[Proof of the lemma] The sets $A_{I, J}$ coincide when the set $I \cup J$ is fixed. This fact can be proved by the third type commutation due to Proposition~\ref{prop:TypesOfComm} with $\zeta$ and $-\zeta$. If $\xi \in A_{I,J}$ we get a transvection $t_{I,J}(\xi) \in H$. Then the following two products also belong to $H$:
\begin{align*}
[t_{I,J}(\xi), \bw{m}t_{j,i}(\zeta)] &= t_{\tilde{I}, J}(\pm \zeta\xi)\cdot t_{I, \tilde{J}}(\pm \zeta\xi)\cdot t_{\tilde{I}, \tilde{J}}(\pm\zeta^2\xi)\\
[t_{I,J}(\xi), \bw{m}t_{j,i}(-\zeta)] &= t_{\tilde{I}, J}(\mp \zeta\xi)\cdot t_{I, \tilde{J}}(\mp \zeta\xi)\cdot t_{\tilde{I}, \tilde{J}}(\pm \zeta^2\xi).
\end{align*}
This implies that the product of two factors on the right-hand sides $t_{\tilde{I}, \tilde{J}}(\pm 2\zeta^2\xi)$ belongs to $H$.

It can be easily proved the set $I\cup J$ can be changed by the second type commutations. For example, the set $I_{1}\cup J_{1} = \{1,2,3,4,5,6\}$ can be replaced by the set $I_{2} \cup J_{2} = \{1,2,3,4,5,7\}$ as follows
$$[t_{123,456}(\xi), \bw{3}t_{6,7}(\zeta)] = t_{123, 457}(\xi\zeta).$$
\end{proof}

\begin{proof}[Proof of Proposition~$\ref{prop:LevelsEqualGeneral}$]
Arguing as above, we see that the sets $A_{I,J}$ and $A_{K,L}$ coincide in the case $I\cap J = K\cap L$, where $n_{1} = n - |I\cap J| \geq 2\cdot m - 2 \cdot |I\cap J| = 2 \cdot m_{1}$.

In a general case, we can prove the statement by both the second and the third types commutations. Let us give an example of this calculation with replacing the set $I \cap J = \{1,2\}$ by the set $\{1,5\}$.

Let $t_{123,124}(\xi) \in H$. So we have $[t_{123,124}(\xi),  \bw{3}t_{2,5}(\zeta) ] = t_{123,145}(-\xi\zeta) \in H$. We commute this transvection with the element $\bw{3}t_{5,2}(\zeta_1)$. Then the transvection $t_{135,124}(-\zeta_1^2\xi\zeta)$ belongs to $H$ as well as the product $t_{123,124}(\xi\zeta\zeta_1)\cdot t_{135, 145}(-\zeta_1\xi\zeta)\in H$. From the latter inclusion we can see $t_{135, 145}(-\zeta_1\xi\zeta)\in H$ and $I\cap J = \{1,5\}$.

To prove all $A_{I,J}$ are ideals in $R$ it is sufficient to commute any elementary transvection with exterior transvections with $\zeta$ and $1$:
$$t_{I,J}(\xi\zeta)=[t_{I,J}(\xi),\bw{m}t_{j,i}(\zeta),\bw{m}t_{i,j}(\pm 1)]\in H.$$
\end{proof}

Let $t_{I,J}(\xi)$ be an elementary transvection. Let us define a \textit{height} of $t_{I,J}(\xi)$ (more abstractly, of the  pair $(I,J)$) as the cardinality of the set $I\cap J$:
$$\height(t_{I,J}(\xi))=\height(I,J)=|I\cap J|.$$
This combinatorial characteristic plays the same role as the distance function $d(\lambda,\mu)$ for roots $\lambda$ and $\mu$ on a weight diagram of a root system. Now Proposition~\ref{prop:LevelsEqualGeneral} can be rephrased as follows.
Sets $A_{I,J}$ and $A_{K,L}$ coincide for the same heights: $A_{I,J} = A_{K,L} =A_{|I\cap J|}$. Suppose that the height of $(I,J)$ is larger than the height of $(K,L)$, then using Proposition~\ref{prop:TypesOfComm}, we get $A_{I,J} \leq A_{K,L}$.

Summarizing the above arguments, we have the height grading:
$$A_{0} \geq A_{1} \geq A_{2} \geq \dots \geq A_{m-2} \geq A_{m-1}.$$

The following result proves a coincidence of the sets $\{A_{k}\}_{k=0\dots m-1}$.
\begin{proposition}\label{prop:IdealEqual}
The ideals $A_{k}$ coincide for $n \geq 3m$. More accurately, the inverse inclusion $A_k\leq A_{k+1}$ takes place if $n\geq 3m-2k$.
\end{proposition}
\begin{proof}
The statement can be proved by the double third type commutation as follows. Let $\xi\in A_k$, i.\,e., a transvection $t_{I,J}(\xi)\in H$ for $\height(t_{I,J}(\xi))=k$. By the third type commutation with a transvection $\bw{m}t_{j,i}(\zeta)$, we have $t_{\tilde{I}, J}(\pm \zeta\xi)\cdot t_{I, \tilde{J}}(\pm \zeta\xi)\in H$. Let us consider an analogous commutator with a specifically chosen transvections $t_{I_1,J_1}(\xi)\in H$ and $\bw{m}t_{j_1,i_1}(\zeta_1)$. We get that $t_{\tilde{I}_1, J_1}(\pm \zeta_1\xi)\cdot t_{I_1, \tilde{J}_1}(\pm \zeta_1\xi)\in H$. The final step is to commute the latter products.

The choice of transvections goes in the way such that the final commutator (initially of the form $[ab,cd]$) equals an elementary transvection. This choice is possible due to the condition $n \geq 3m-2k$.

Let us give a particular example of such calculations for the case $m=4$. This calculation could be easily generalized. The first three steps below correspond to the inclusions $A_{0} \leq A_{1}$, $A_{1} \leq A_{2}$, and $A_{2} \leq A_{3}$ respectively. We emphasize that the ideas of the proof of all three steps are \textbf{completely} identical. The difference has to do only with a choice of the appropriate indices. We replace the numbers $10$, $11$, $12$ with the letters $\alpha$, $\beta$, $\gamma$ respectively.

\begin{enumerate}
\item Let $\xi\in A_0$. Consider the mutual commutator
$$\bigl[[t_{1234,5678}(\xi), \bw{4}t_{8,4}(\zeta)],[t_{49\alpha\beta,123\gamma}(\xi), \bw{4}t_{\gamma,4}(\zeta_1)]\bigr]\in H.$$
It is equal to the commutator
$$[t_{1234,4567}(-\xi\zeta)\cdot t_{1238,5678}(-\zeta\xi),t_{49\alpha\beta,1234}(\xi\zeta_1)\cdot t_{9\alpha\beta\gamma,123\gamma}(\zeta_1\xi)]\in H,$$
which is a transvection $t_{49\alpha\beta,4567}(\xi^2\zeta_1\zeta)\in H$. As the result, $A_0\leq A_1$.
\item For $\xi\in A_1$ consider similar commutator
$$\bigl[[t_{1234,1567}(\xi), \bw{4}t_{7,4}(\zeta)],[t_{1489,123\alpha}(\xi), \bw{4}t_{\alpha,4}(\zeta_1)]\bigr]\in H.$$
Thus,
$$[t_{1234,1456}(\xi\zeta)\cdot t_{1237,1567}(-\zeta\xi),t_{1489,1234}(\xi\zeta_1)\cdot t_{189\alpha,123\alpha}(-\zeta_1\xi)]\in H.$$
Again this commutator is equal to $t_{1489,1456}(-\xi^2\zeta_1\zeta)\in H$, i.\,e., $A_1\leq A_2$.
\item Finally, let $\xi\in A_2$. Consider the commutator
$$\bigl[[t_{1234,1256}(\xi), \bw{4}t_{6,4}(\zeta)],[t_{1248,1237}(\xi), \bw{4}t_{7,4}(\zeta_1)]\bigr]\in H.$$
It is equal to the commutator
$$[t_{1234,1245}(-\xi\zeta)\cdot t_{1236,1256}(-\zeta\xi),t_{1248,1234}(\xi\zeta_1)\cdot t_{1278,1237}(-\zeta_1\xi)]\in H,$$
which is an elementary transvection $t_{1248,1245}(\xi^2\zeta_1\zeta)\in H$. Thus, $A_2\leq A_3$.
\end{enumerate}
\end{proof}

We proved that all ideals $A_i$ coincide for a large enough $n$. However, the following proposition shows relations between the ideals without this restriction. Recall that a residue $\opn{res}$ of an exterior transvection $\bw{m}t_{i,j}(\xi)$ equals the binomial coefficient $\binom{n-2}{m-1}$.
\begin{proposition}\label{prop:Relations}
For ideals $\{A_0,\dots,A_{m-1}\}$ the following relations hold:
\begin{gather*}
A_k\leq A_{k+1}, \text{ for } n\geq 3m-2k;\\
A_{0} \geq A_{1} \geq A_{2} \geq \dots \geq A_{m-2} \geq A_{m-1};\\
\opn{res}\cdot A_{m-2}\leq A_{m-1}.
\end{gather*}
\end{proposition}
Note that we have not included the relations from the notion of $D$-net $A_{I,J}\cdot A_{J,K}\leq A_{I,K}$, since they hold for any net of ideals by the definition.
\begin{proof}
The first two series of relations have already been proved. Therefore, we must only prove that $\opn{res}\cdot A_{m-2}\leq A_{m-1}$. Again we will use the third type commutation.

Let $\xi\in A_{m-2}$,i.\,e., for any indices $I,J$ with $\height(I,J)=m-2$ a transvection $t_{I,J}(\xi)\in H$. Note that if $i\in I, j\in J$, then in the commutator 
$$[t_{I,J}(\xi),\bw{m}t_{j,i}(\zeta)]=t_{\tilde{I}, J}(\pm \zeta\xi)\cdot t_{I, \tilde{J}}(\pm \zeta\xi)\cdot t_{\tilde{I}, \tilde{J}}(\pm\zeta^2\xi)$$
the transvection $t_{\tilde{I}, \tilde{J}}(\pm\zeta^2\xi)$ belongs to the group $H$. Indeed, the height of indices $\tilde{I}=I\backslash i \cup j$ и $\tilde{J} = J\backslash j \cup i$ coincide with the height of $I,J$. At the same time the height of $\tilde{I}, J$ and $I, \tilde{J}$ equals $m-1$. Thus $t_{\tilde{I}, J}(\pm \zeta\xi)\cdot t_{I, \tilde{J}}(\pm \zeta\xi)\in H$ for all indices $I,J$ with $\height(I,J)=m-2$ and all different $i\in I, j\in J$.

Consider $\bw{m}t_{1,2}(\xi\zeta)\in H$, where $\zeta\in R$. By the definition of exterior transvections~\eqref{eq:m}, we have $\bw{m}t_{1,2}(\xi\zeta)=\prod\limits_{L} t_{L\cup 1,L\cup 2}(\xi\zeta)$. The proof is to consistently reduce the number of factors in the product by multiplication $\bw{m}t_{1,2}(\xi\zeta)$ on suitable transvections $t_{\tilde{I}, J}(\pm \zeta\xi)\cdot t_{I, \tilde{J}}(\pm \zeta\xi)\in H$. Finally, we get an elementary transvection $t_{P\cup 1,P\cup 2}(c\xi\zeta)$, where the height of indices equals $m-1$ and the coefficient $c$ equals $\binom{n-2}{m-1}$.

Let us give an example such argument for the exterior cube of the elementary group of dimension 5. Take $\xi\in A_1,\zeta\in R$, $\bw{3}t_{1,2}(\xi\zeta)=t_{134,234}(\xi\zeta)t_{135,235}(\xi\zeta)t_{145,245}(\xi\zeta)$.\\
First, consider the commutator $[t_{134,245}(\xi),\bw{3}t_{5,3}(\zeta)]\in H$. As we mentioned above, the matrix $z_1:=t_{134,234}(-\xi\zeta)t_{145,245}(\xi\zeta)\in H$. Thus, 
$$\bw{3}t_{1,2}(\xi\zeta)\cdot z_1=t_{135,235}(\xi\zeta)t_{145,245}(2\xi\zeta)\in H.$$
To get an elementary transvection, consider one more commutator\\ 
$[t_{135,245}(\xi),\bw{3}t_{4,3}(-\zeta)]\in H$. Then the matrix $z_2:=t_{145,245}(\xi\zeta)t_{135,235}(-\xi\zeta)\in H$. It remains to multiply $\bw{3}t_{1,2}(\xi\zeta)$ and $z_1z_2$. We get the transvection $t_{145,245}(3\xi\zeta)\in H$. Therefore, $3\xi\zeta\in A_{2}$.
\end{proof}

As usual, the set $A = A_{I,J}$ is called a \textit{level} of an overgroup $H$. For level computation we need an alternative description of the relative elementary group.

\begin{lemma}\label{lem:AlterRelatForm}
For any ideal $A\trianglelefteq R$, we have
$$\E_N(A)^{\bw{m}\E_n(R)}=\E_N(R,A),$$
where by definition $\E_N(R,A)=\E_N(A)^{\E_N(R)}$.
\end{lemma}
\begin{proof}
Clearly, the left-hand side is contained in the right-hand side. The proof of the inverse inclusion goes by induction on the height of $(I,J)$. By Vaserstein--Suslin's lemma~\cite{VasSusSerre} it is sufficient to check the matrix $z_{I,J}(\xi,\zeta)$ to belong to $F:=\E_N(A)^{\bw{m}\E_n(R)}$ for any $\xi \in A$, $\zeta \in R$.

In the base case $|I\cap J|=m-1$, the inclusion is obvious:
$$z_{I,J}(\xi,\zeta)\cdot t_{I,J}(-\xi)=[t_{J,I}(\zeta),t_{I,J}(\xi)]=\left[\bw{m}t_{j_1,i_1}(\zeta),t_{I,J}(\xi) \right]\in F.$$

Now, let us consider the general case $|I\cap J|=p$, i.\,e., $I=k_1\dots k_{p}i_1\dots i_{q}$ and $J=k_1\dots k_pj_1\dots j_q$. For the following calculations we need two more sets $V:=k_1\dots k_pi_1\dots i_{q-1}j_q$ and $W:=k_1\dots k_pj_1\dots j_{q-1}i_q$.\\
Firstly, we express $t_{I,J}(\xi)$ as a commutator of elementary transvections,
$$z_{I,J}(\xi,\zeta)=^{t_{J,I}(\zeta)}t_{I,J}(\xi)=^{t_{J,I}(\zeta)}[t_{I,V}(\xi),t_{V,J}(1)].$$
Conjugating the arguments of the commutator by $t_{J,I}(\zeta)$, we get
$$[t_{J,V}(\zeta\xi)t_{I,V}(\xi),t_{V,I}(-\zeta)t_{V,J}(1)]=:[ab,cd].$$
Next, we decompose the right-hand side with a help of the formula 
$$[ab,cd] ={}^a[b,c]\cdot {}^{ac}[b,d]\cdot[a,c]\cdot {}^c[a,d],$$
and observe the exponent $a$ to belong to $\E_N(A)$, so can be ignored. Now a direct calculation, based upon the Chevalley commutator formula, shows that
\begin{align*}
[b,c]&=[t_{I,V}(\xi),t_{V,I}(-\zeta)]\in F \textrm{ (by the induction step for the height } m-1);\\
^c[b,d]&=^{t_{V,I}(-\zeta)}[t_{I,V}(\xi),t_{V,J}(1)]=t_{V,W}(\xi\zeta^2)t_{I,W}(-\xi\zeta)\cdot\; ^{\bw{m}t_{j_q,i_q}(-\zeta)}t_{I,J}(\xi);\\
[a,c]&=[t_{J,V}(\zeta\xi),t_{V,I}(-\zeta)]=t_{J,I}(-\zeta^2\xi);\\
^c[a,d]&=^{t_{V,I}(-\zeta)}[t_{J,V}(\zeta\xi),t_{V,J}(1)]=\\
=t_{J,W}(-\xi\zeta^2(1+\xi\zeta))&t_{V,W}(-\xi\zeta^2)\cdot\;^{\bw{m}t_{j_q,i_q}(-\zeta)}t_{J,V}(\xi\zeta)\cdot\;^{\bw{m}t_{j_q,i_q}(-\zeta)}z_{J,V}(-\zeta\xi,1)\in F,\\
\hspace{2cm}&\hspace{4cm}\textrm{ (by the induction step for the height }p+1)
\end{align*}
where all factors on the right-hand side belong to $F$.
\end{proof}

\begin{remark}
Since we do not use the coincidental elements of $I$ and $J$, we also can prove this Lemma by induction on  $|I\backslash J| = |J\backslash I| = 1/2 \cdot |I \triangle J|$. Then we can assume that $m$ is an arbitrarily large number (mentally, $m =\infty$).
\end{remark}

\begin{corollary}\label{cor:CorolOfL6}
Suppose $A$ be an arbitrary ideal of the ring $R$; then
$$\bw{m}\E_n(R)\cdot\E_N(R,A)=\bw{m}\E_n(R)\cdot\E_N(A).$$
\end{corollary}

Summarizing Proposition~\ref{prop:IdealEqual} and Lemma~\ref{lem:AlterRelatForm}, we get the main result of the paper for the general case.
\LevelForm*

\subsection{Normalizer of \texorpdfstring{$\E\bw{m}\E_n(R,A)$}{TEXT}}\label{subsec:normgeneral}

In this section, we describe a normalizer of the lower bound for a group $H$.
\begin{lemma}\label{lem:PerfectForm}
Let $n\geq 3m$. The group $\E\bw{m}\E_n(R,A) := \bw{m}\E_n(R)\cdot\E_N(R,A)$ is perfect for any ideal $A\trianglelefteq R$.
\end{lemma}
\begin{proof}
It is sufficient to verify all generators of the group $\bw{m}\E_n(R)\cdot\E_N(R,A)$ to lie in its commutator subgroup, which will be denoted by $F$. The proof goes in two steps.
\begin{itemize}
\item For the transvections $\bw{m}t_{i,j}(\zeta)$ this follows from the Cauchy--Binet homomorphism:
$$\bw{m}t_{i,j}(\zeta)=\bw{m}([t_{i,h}(\zeta),t_{h,j}(1)])=\left[\bw{m}t_{i,h}(\zeta),\bw{m}t_{h,j}(1)\right].$$
\item For elementary transvections $t_{I,J}(\xi)$ this can be done as follows. Suppose that $I\cap J=K = k_1\dots k_p$, where $0\leq p\leq m-1$, i.\,e., $I=k_1\dots k_pi_1\dots i_q$ and $J=k_1\dots k_pj_1\dots j_q$. As in Lemma~\ref{lem:AlterRelatForm}, we define a set $V=k_1\dots k_{p}j_1\dots j_{q-1}i_q$. And then 
$$t_{I,J}(\xi)=[t_{I,V}(\xi),t_{V,J}(1)]=\left[t_{I,V}(\xi), \bw{m}t_{i_q,j_q}(\pm 1) \right],$$
so we get the required.
\end{itemize}
\end{proof}

Let, as above, $A\trianglelefteq R$, and let $R/A$ be the factor-ring of $R$ modulo $A$. Denote by $\rho_A: R\longrightarrow R/A$ the canonical projection sending $\lambda\in R$ to $\bar{\lambda}=\lambda+A\in R/A$. Applying the projection to all entries of a matrix, we get the reduction homomorphism 
$$\begin{array}{rcl}
\rho_{A}:\GL_{n}(R)&\longrightarrow& \GL_{n}(R/A)\\
a &\mapsto& \overline{a}=(\overline{a}_{i,j})
\end{array}$$
The kernel of the homomorphism $\rho_{A}$, $\GL_{n}(R,A)$, is called the \textit{principal congruence subgroup in $\GL_{n}(R)$ of level $A$}. Now, let $\CC_n(R)$ be the center of the group $\GL_{n}(R)$, consisting of the scalar matrices
$\lambda e, \lambda \in R^{*}$. The full preimage of the center of $\GL_{n}(R/A)$, denoted by $\CC_n(R,A)$, is called the \textit{full congruence subgroup of level $A$}. The group $\CC_n(R,A)$ consists of all matrices congruent to a scalar matrix modulo $A$. We further concentrate on a study of the full preimage of the group $\bw{m}\GL_{n}(R/A)$:
$$\CC\bw{m}\GL_{n}(R,A)=\rho_{A}^{-1}\bigl(\bw{m}\GL_{n}(R/A)\bigr).$$
A key point in a reduction modulo an ideal is the following standard commutator formula, proved by Leonid Vaserstein~\cite{VasersteinGLn}, Zenon Borevich, and Nikolai Vavilov~\cite{BV84}.
$$\bigl[\E_n(R),\CC(n,R,A)\bigr]=\E_n(R,A) \textrm{ for a commutative ring }R \textrm{ and } n\geq 3.$$

Finally, we are ready to state the \textit{level reduction} result.
\LevelReduction*
\begin{proof}
In the proof by $N$ we mean $N_{\GL_{N}(R)}$.

Since $\E_N(R,A)$ and $\GL_{N}(R,A)$ are normal subgroups in $\GL_{N}(R)$, we see
$$N\bigl(\underbrace{\E\bw{m}\E_n(R,A)}_{=\bw{m}\E_n(R)\E_N(R,A)}\bigr)\leq N\bigl(\E\bw{m}\E_n(R,A)\GL_{N}(R,A)\bigr)=\CC\bw{m}\GL_{n}(R,A). \eqno(1)$$
Note that the latter equality is due to the normalizer functoriality:
$$N\left(\E\bw{m}\E_n(R,A)\GL_{N}(R,A)\right)=N\left(\rho_A^{-1}\left(\bw{m}\E_n(R/A)\right)\right)=\rho_A^{-1}\left(N \left(\bw{m}\E_n(R/A)\right)\right)=\rho_A^{-1}\left(\bw{m}\GL_{n}(R/A)\right).$$
In particular, using $(1)$, we get
$$\left[\CC\bw{m}\GL_{n}(R,A),\E\bw{m}\E_n(R,A)\right]\leq \E\bw{m}\E_n(R,A)\GL_{N}(R,A).\eqno(2)$$
On the other hand, it is completely clear $\E\bw{m}\E_n(R,A)$ to be normal in the right-hand side subgroup. Indeed, it is easy to prove the following stronger inclusion:
$$\left[\bw{m}\GL_{n}(R)\GL_{N}(R,A),\E\bw{m}\E_n(R,A)\right]\leq \E\bw{m}\E_n(R,A).\eqno(3)$$
To check this, we consider a commutator of the form
$$[xy,hg],\qquad x\in \bw{m}\GL_{n}(R), y\in \GL_{N}(R,A), h\in \bw{m}\E_n(R), g\in \E_N(R,A).$$
Then $[xy, hg]={}^x[y,h] \cdot [x,h] \cdot {}^h[xy,g]$. We need to prove all factors on the right-hand side to belong to $\E\bw{m}\E_n(R,A)$. Right away, the second factor lies in the group $\E\bw{m}\E_n(R,A)$. For the first commutator, we should consider the following inclusions:
$${}^{\bw{m}\GL_{n}(R)}\left[\GL_{N}(R,A),\bw{m}\E_n(R)\right]\leq \Bigl[
\underbrace{{}^{\bw{m}\GL_{n}(R)}\GL_{N}(R,A)}_{=\GL_{N}(R,A)},\underbrace{{}^{\bw{m}\GL_{n}(R)}\bw{m}\E_n(R)}_{=\bw{m}\E_n(R)}\Bigr]\leq \E\bw{m}\E_n(R,A).$$ 
The element $h\in \bw{m}\E_n(R)$, so we ignore it in conjugation. The third commutator lies in  $\E\bw{m}\E_n(R,A)$ due to the following inclusion. 
$$\left[\bw{m}\GL_{n}(R)\GL_{N}(R,A),\E_N(R,A)\right]\leq\left[\GL_{N}(R),\E_N(R,A)\right]=\E_N(R,A).$$
Now if we recall $(2)$ and $(3)$, we get
$$\left[\CC\bw{m}\GL_{n}(R,A),\E\bw{m}\E_n(R,A),\E\bw{m}\E_n(R,A)\right]\leq \E\bw{m}\E_n(R,A).\eqno(4)$$
To invoke the Hall–-Witt identity, we need a slightly more precise version of the latter
inclusion:
$$\left[\left[\CC\bw{m}\GL_{n}(R,A),\E\bw{m}\E_n(R,A)\right],\left[\CC\bw{m}\GL_{n}(R,A),\E\bw{m}\E_n(R,A)\right]\right]\leq \E\bw{m}\E_n(R,A).\eqno(5)$$
Observe that by formula $(2)$ we have already checked the left-hand side to be generated by the commutators
of the form 
$$[uv,[z,y]], \text{ where } u,y\in \E\bw{m}\E_n(R,A), v\in \GL_N(R,A), z\in \CC\bw{m}\GL_n(R,A).$$
However,
$$[uv,[z,y]]={}^u[v,[z,y]]\cdot[u,[z,y]],$$
By formula $(4)$ the second commutator belongs to $\E\bw{m}\E_n(R,A)$, whereas by $(5)$ the first is an element of $\left[\GL_{N}(R,A),\E_N(R)\right]\leq \E_N(R,A)$. 

Now we are ready to finish the proof. By the previous lemma, the group $\E\bw{m}\E_n(R,A)$
is perfect, and thus, it suffices to show $[z,[x,y]]\in \E\bw{m}\E_n(R,A)$ for all $x,y\in \E\bw{m}\E_n(R,A), z\in \CC\bw{m}\GL_{n}(R,A)$. Indeed, the Hall–-Witt identity yields
$$[z,[x,y]]={}^{xz}[[z^{-1},x^{-1}],y]\cdot{}^{xy}[[y^{-1},z],x^{-1}],$$
where the second commutator belongs to $\E\bw{m}\E_n(R,A)$ by $(4)$.
Removing the conjugation by $x\in \E\bw{m}\E_n(R,A)$ in the first commutator and carrying the conjugation by $z$ inside
the commutator, we see that it only remains to prove the relation $[[x^{-1},z],[z,y]y]\in \E\bw{m}\E_n(R,A)$. Indeed,
$$[[x^{-1},z],[z,y]y]=[[x^{-1},z],[z,y]]\cdot{}^{[z,y]}[[x^{-1},z],y],$$
where both commutators on the right--hand side belong to $\E\bw{m}\E_n(R,A)$ by formulas $(4)$ and $(5)$, and moreover,
the conjugating element $[z,y]$ in the second commutator is an element of the group $\E\bw{m}\E_n(R,A)\GL_{N}(R,A)$, and thus by $(3)$, normalizes $\E\bw{m}\E_n(R,A)$.
\end{proof}

\bibliographystyle{unsrt}
\bibliography{english}

\end{document}